\documentclass[12pt,a4paper]{amsart}
\usepackage{amsmath,amsfonts,amssymb,amsthm,amscd}
\usepackage{pst-node,pstricks,pst-tree}
\usepackage{pstricks-add}
\usepackage[english]{babel}

\input{comment.sty}
\includecomment{MM}
\includecomment{CL}
\renewcommand{\a}{\alpha}
\renewcommand{\b}{\beta}
\newcommand{\g}{\gamma}
\newcommand{\G}{\Gamma}
\renewcommand{\d}{\delta}
\newcommand{\D}{\Delta}

\renewcommand{\l}{\lambda}

\newcommand{\n}{\nu}
\renewcommand{\o}{\omega}
\renewcommand{\O}{\Omega}
\renewcommand{\r}{\rho}
\newcommand{\s}{\sigma}
\renewcommand{\SS}{\Sigma}
\renewcommand{\t}{\tau}

\newcommand{\e}{\varepsilon}
\newcommand{\f}{\varphi}
\newcommand{\F}{\Phi}
\newcommand{\x}{\xi}

\newcommand{\z}{\zeta}

\newcommand{\Fc}{{\mathcal F}}
\newcommand{\Oc}{{\mathcal O}}
\newcommand{\C}{{\mathbb C}}
\newcommand{\N}{{\mathbb N}}
\newcommand{\Z}{{\mathbb Z}}

\newtheorem{theorem}{Theorem}[section]
\newtheorem{proposition}[theorem]{Proposition}
\newtheorem{lemma}[theorem]{Lemma}
\newtheorem{corollary}[theorem]{Corollary}

\newtheorem{property}[theorem]{Property}
\newtheorem{fact}[theorem]{Fact}

\theoremstyle{definition}
\newtheorem{definition}[theorem]{Definition}

\theoremstyle{remark}
\newtheorem{remark}[theorem]{Remark}

\newtheorem{examples}[theorem]{Examples}

\begin{document}

\title[Growth behaviors in the range $e^{r^\a}$]
      {Growth behaviors in the range $e^{r^\a}$}
\author[Brieussel]{J\'er\'emie Brieussel \\ \\ Institut de Math\'ematiques \\ Rue Emile-Argand 11 \\ CH-2000 Neuch\^atel \\ Switzerland}
\email{jeremie.brieussel@gmail.com}

\date{June 27, 2012. Supported by Swiss NSF grant 20-126689}

\begin{abstract}
For every $\a \leq \b$ in a left neighborhood $[\a_0,1]$ of $1$, a group $G(\a,\b)$ is constructed, the growth function of which satisfies $\limsup \frac{\log \log b_{G(\a,\b)}(r)}{\log r}=\a$ and $\liminf \frac{\log \log b_{G(\a,\b)}(r)}{\log r}=\b$. When $\a=\b$, this provides an explicit uncountable collection of groups with growth functions strictly comparable. On the other hand, oscillation in the case $\a < \b$ explains the existence of groups with non comparable growth functions. Some period exponents associated to the frequency of oscillation provide new group invariants. 
\end{abstract}

\maketitle

\section{Introduction}

The growth function $b_{\G,S}(r)=|S^r|$ of a group $\G$ with finite generating set $S$ was introduced by Milnor \cite{Mil} in relation with Riemannian geometry. The class $b_\G(r)$ of $b_{\G,S}(r)$ under the equivalence relation associated to the order $f(r) \leq g(Cr)$ for some $C$ (written $f \precsim g$) is independant of the generating set $S$, so that $b_\G(r)$ is a group invariant.

For many groups, e.g. those containing a free semigroup, the growth function is exponential. However, the growth function of a nilpotent group $\G$ is polynomial $b_\G(r) \approx r^{d(\G)}$ where $d(\G)=\sum k. \textrm{rank}(\G_k/\G_{k+1})$ is the algebraic degree of nilpotency of $\G=\G_1$ associated to the filtration $\G_{k+1}=[\G_k,\G]$ (\cite{Bas}, \cite{Gui}, \cite{Wol}). Conversely, Gromov proved that polynomial growth implies virtual nilpotency (\cite{Gro}, see also \cite{Kle} and \cite{ST} for an explicit version applying to finite groups). This implies in particular that polynomial growth functions are indexed by integers $d(\G)$ and any two are always comparable for $\precsim$.

In the eighties, Grigorchuk has shown some groups have intermediate growth, i.e. faster than polynomial and slower than exponential. In \cite{Gri1}, he considers a family indexed by a Cantor set $\{0,1,2\}^\N$ of groups $G_\o$ acting on a binary rooted tree. These groups are commensurable with the infinite torsion groups constructed earlier by Aleshin in \cite{Ale}. Many groups in these families satisfy growth inequalities of the form $e^{r^\a} \precsim b_{G_\o}(r) \precsim e^{r^\b}$ for exponents $\frac{1}{2} \leq \a < \b < 1$. On the other hand, for some sequences $\o$, the growth of $G_\o$ is ``close to'' $e^r$. Grigorchuk also proved the existence of uncountable antichains of growth functions (i.e. collections of pairwise non comparable such functions).

Recently, Bartholdi and Erschler have computed the intermediate growth functions of some groups related to the group $G_{(012)^\infty}$ (see \cite{BE}). More precisely, for $\G_0= \Z/2\Z \wr_X G_{(012)^\infty}$ and $\G_{k+1}= \G_k \wr_X G_{(012)^\infty}$, there are explicit exponents $\a_k <1$ accumulating to $1$, such that their growth functions satisfy $b_{\G_k}(r) \approx e^{r^{\a_k}}$.

The purpose of the present article is to draw a panorama of growth behaviors in the range $e^{r^\a}$. The two main points are that on the one hand there is a neighborhood of $1$ in which any $\a$ is the growth exponent of some group, raising an explicit uncountable family of groups for which the growth functions are strictly comparable, and on the other hand, there are groups the growth function of which oscillates between two distinct exponents $\a < \b$, which explains non comparison phenomena. More precisely:

\begin{theorem}\label{mainthm}
Let $\eta\approx 0.8105$ be the real root of $X^3+X^2+X-2$ and $\a_0=\frac{\log 2}{\log 2-\log\eta}\approx 0.7674$. Then
for any $\a_0 \leq \a \leq \b \leq 1$, there exists a group $G(\a,\b)$ such that:
$$\liminf \frac{\log \log b_{G(\a,\b)}(r)}{\log r} = \a \textrm{ and $\limsup \frac{\log \log b_{G(\a,\b)}(r)}{\log r} = \b$.} $$
In particular, there exists a group $G(\a)$ such that $\lim \frac{\log \log b_{G(\a)}(r)}{\log r}=\a$.
\end{theorem}

The groups $G(\a,\b)$ will be explicitely described as $F \wr_X G_{\o}$ for appropriate sequence $\o=\o(\a,\b)$. Note that the group $G(\a_0)$ is precisely the group $\G_0=\Z/2\Z \wr_X G_{(012)^\infty}$ considered in \cite{BE}. Also a better study of oscillation phenomena provides uncountable antichains of growth functions satisfying a uniform upper bound $e^{r^\b}$ for any $\b > \a_0$.

The strategy to prove theorem \ref{mainthm} is to construct first the groups $G(\a)$ by appropriate sequences $\o(\a)$. The construction of the groups $G(\a,\b)$ with oscillating growth for $\a<\b$ then follows from the fact that the asymptotic behavior of the growth of $G_\o$ depends only on the asymptotic of $\o$, whereas the first values of the growth function depend only on the first values of $\o$. This permits to construct groups with different growth behaviors at different scales.

In order to ease notation, adopt the following:
\begin{definition}
Given a finitely generated group $G$, the {\it upper logarithmic growth exponent} $\overline{\a}(G)$ and the {\it lower logarithmic growth exponent} $\underline{\a}(G)$ are real numbers in $[0,1]$ defined as:
$$\overline{\a}(G)= \limsup \frac{\log \log b_G(r)}{\log r} \textrm{ and $\underline{\a}(G)= \liminf \frac{\log \log b_G(r)}{\log r}$.} $$
In case of equality, call {\it logarithmic growth exponent} the number $\a(G)=\overline{\a}(G)=\underline{\a}(G)$.
\end{definition}
For submultiplicative functions, inequality $b(Cr)\leq b(r)^C$ implies:
$$\frac{\log \log b(Cr)}{\log r} \leq \frac{\log \log b(r)}{\log r} + \frac{\log C}{\log r}, $$
so that the logarithmic growth exponents of groups are independent of the choice of a particular representative $b_{\G,S}(r)$, i.e. the choice of generating set. Note that if $b_G(r) \simeq e^{r^\a}$, then $\a(G)=\a$ but the converse is not true, as shown by functions $e^{r^\a(\log r)^p}$ for any value of $p$. In particular, the growth functions of the groups studied here are not computed, but only their logarithmic growth exponents.

The article is structured as follows. Sections 2 and 3 are devoted to the description of the involved groups $\G_\o$, and in particular the notion of activity of a representative word. Section 4 presents the three main tools of estimation for growth. The activity of words is studied in section 5 to derive precise growth estimates, i.e. construct groups with a given logarithmic growth exponent. Oscillation phenomena are studied in sections 7 and 8, which permits to explain the existence of antichains of growth functions. Some explicit estimates on the frequency of oscillation are given. A few comments and some questions conclude the article. 

Note that after this work was first submitted, further computations of growth functions of similar groups have been done in \cite{KP} and \cite{BE2}. Also similar estimates on entropy of random walks have been done in \cite{Bri3}.

\section{The groups involved}

\subsection{Definition}\label{definition} Following Grigorchuk \cite{Gri1}, associate to each given sequence $\o=\o_0\o_1\o_2\dots$ in $\{0,1,2\}^\N$ a group $G_\o$ of automorphism of a binary rooted tree $T$, generated by four elements $G_\o=\langle a,b_\o,c_\o,d_\o\rangle$, defined via the wreath product isomorphism:
\begin{eqnarray}Aut(T) \simeq Aut(T) \wr S_2 = (Aut(T) \times Aut(T)) \rtimes S_2, \label{wr} \end{eqnarray}
where $S_2$ acts on the product by permuting components. The generator $a=(1,1)\e$, where $\e$ is non-identity in $S_2$, is independent of $\o$ and only acts at the root of $T$. The three other generators are defined recursively by:
\begin{eqnarray} b_\o=(u^b(\o_0),b_{\s \o}), c_\o=(u^c(\o_0),c_{\s \o}), d_\o=(u^d(\o_0),d_{\s \o}), \label{generateurs}\end{eqnarray}
where $\s$ is the shift of sequence $\s \o=\o_1\o_2\dots$ and:  
\begin{eqnarray} u^b \left(\begin{array}{c} 0 \\ 1 \\ 2 \end{array} \right)=\left(\begin{array}{c} a \\ a \\ id \end{array} \right),u^c \left(\begin{array}{c} 0 \\ 1 \\ 2 \end{array} \right)=\left(\begin{array}{c} a \\ id \\ a \end{array} \right),u^d \left(\begin{array}{c} 0 \\ 1 \\ 2 \end{array} \right)=\left(\begin{array}{c} id \\ a \\ a \end{array} \right). \label{defbcd} \end{eqnarray}
The group $G_\o$ is defined by the sequence $\o$ which rules the embeddings $G_\o \hookrightarrow G_{\s\o}\wr S_2$. The following relations are easily checked:
\begin{eqnarray} a^2=b_\o^2=c_\o^2=d_\o^2=b_\o c_\o d_\o=id. \label{relationsbasic} \end{eqnarray}
In particular, the group generated by $b_\o,c_\o,d_\o$ is a Klein group $V=S_2 \times S_2$ and each of the four generators has order 2 (unless $\o$ is constant), so they generate $G_\o$ as a quotient semigroup of $\O_\o=\{a,b_\o,c_\o,d_\o\}^\ast$, the free semigroup of words in the generators with concatenation as product. Also note that conjugating by $a$ exchanges the components on the two subtrees, in particular:
\begin{eqnarray}ab_\o a=(b_{\s \o},u^b(\o_0)), ac_\o a=(c_{\s \o},u^c(\o_0)), ad_\o a=(d_{\s \o},u^b(\o_0)). \label{generateursconjugues}\end{eqnarray}

Now following \cite{BE}, let $\r=1^\infty \in \partial T$ be the rightmost geodesic ray out of the root of $T$. Note that $b_\o,c_\o,d_\o$ fix $\r$ independently of $\o$. Denote $X=\r G_\o$ the right orbit of $\r$ under $G_\o$. The permutational wreath product of $G_\o$ and another group $F$ over $X$ is the group:
$$\G_\o = F \wr_X G_\o = \left( \SS_X F \right) \rtimes G_\o, $$
where $\SS_X F$ is the group of finitely supported functions $\f:X \rightarrow F$, on which $G_\o$ acts on the left by $(g.\f)(x)=\f(xg)$, and in particular the supports satisfy $supp(g.\f)=supp(\f)g^{-1}$. The elements are denoted $\f g$ for $\f \in \SS_X F$ and $g \in G_\o$. The computation rule is $(\f_1 g_1)(\f_2 g_2)= (\f_1 (g_1.\f_2)) (g_1g_2)$. Throughout the present article, assume the group $F$ is finite.

As a generating set, use $S_\o=\{ a \} \sqcup \{\f_f v | v \in \{id_{G_\o},b_\o,c_\o,d_\o\},f \in F\}$. Note that $\r v=\r$, so $[\f_f,v]=id_{\G_\o}$ and the set $\{\f_f v\}$ generates a finite subgroup in $\G_\o$, which is abstractly isomorphic to $F \times V$.

\subsection{A short history} The groups $G_\o$ are commensurable with the groups introduced by Aleshin \cite{Ale}, where automata techniques were used to provide a short solution to Burnside's problem. The groups $G_\o$ and especially $G_{(012)^\infty}$ have been widely studied especially since they provide the essentially only known exemples of groups of intermediate growth (\cite{Bar1}, \cite{Bar2}, \cite{BS}, \cite{Bri}, \cite{Ers1}, \cite{Ers2}, \cite{Ers3}, \cite{Gri1}, \cite{Gri2}, \cite{MP}, \cite{Zuk}). In particular, the best known estimates on the growth of $G_{(012)^\infty}$ are:
\begin{theorem}
$$e^{r^{0.5207}}\precsim b_{G_{(012)^\infty}}(r) \precsim e^{r^{\a_0}}. $$
\end{theorem}
The upper bound comes from \cite{Bar1} (see also \cite{MP}) and the lower bound from \cite{Bri} (see also \cite{Bar2}, \cite{Leo}). The estimation on the growth exponents of $G_\o$ is tightly related to the contraction of the length of reduced words $w=(w_0,w_1)$ under the wreath product decomposition (\ref{wr}). If for all reduced words, $|w_0|+|w_1|$ is a large contraction of $|w|$, the upper growth exponent is small. If for {\it all} pairs of reduced words, $|w|$ is a small dilatation of $|w_0|+|w_1|$, the lower growth exponent is big. As it turns out, the study of dilatation of pair of words is delicate to handle, explaining the large gap between the upper and lower exponents of $G_{(012)^\infty}$.

In \cite{BE}, Bartholdi and Erschler have bypassed this problem, considering (among others) the group $F \wr_X G_{(012)^\infty}$, where $F$ is any finite group, for which they prove:
\begin{theorem}\cite{BE}
$$b_{F \wr_X G_{(012)^\infty}}(r) \approx e^{r^{\a_0}}.$$
\end{theorem}
In short, if the upper estimates still apply, the use of permutational wreath product permits to obtain a good lower bound from small dilatation of {\it some} pairs of words. The techniques developed in \cite{BE} are not restricted to the specific sequence $\o=(012)^\infty$, and can provide a good understanding of growth of $\G_\o$ for rotating sequences $\o$, as explained below. The construction of an appropriate sequence $\o(\a)$ or $\o(\a,\b)$ will be the key point to prove Theorem \ref{mainthm}.

\section{A description of the groups}

This section aims at giving description of the group $\G_\o=F\wr_X G_\o$. Useful notions are that of minimal tree and active set of a word. The active set is very much related with the inverted orbit introduced in \cite{BE}.

\begin{lemma}\label{CE}
The group $\G_\o=F \wr_X G_\o$ embeds cannonically into the finite permutational wreath product $\G_{\s\o}\wr S_2$. More precisely, the application $\F$:
\begin{eqnarray*}
\G_\o &\hookrightarrow& \G_{\s\o} \wr S_2 \\
a &\mapsto& (1,1)a \\
v_\o &\mapsto& (u^v(\o_0),v_{\s\o}) \\
\f_f &\mapsto& (1,\f_f)
\end{eqnarray*}
is an injective morphism of groups.
\end{lemma}

Any $\g$ in $\G_\o$ is decomposed $\g=\f g$, with $g \in G_\o$ and $\f: X \rightarrow F$. The classical embedding $G_\o \hookrightarrow G_{\s\o} \wr S_2$ provides a decomposition $g=(g_0,g_1)\s$. Also the boundary of the tree can be decomposed into two components $\partial T = \partial T_0 \sqcup \partial T_1$ with $T_t$ the tree descended from the first level vertex $t$. In particular, the orbit $X$ inherits this decomposition into $X=X_0 \sqcup X_1$. Set $\f_t=\f|_{X_t}$ the restriction of $\f$ to the subset $X_t$ of the orbit $X$. With these notations, the application $\F$ is given by:
$$\F(\g)=(\f_0 g_0,\f_1 g_1)\s \in \G_\o \wr S_2. $$
In order to prove the lemma, it is sufficient to check that $\F(\g\g')=\F(\g)\F(\g')$.

\begin{proof} 
On the one hand, $\g\g'=\f g\f' g'=\f (g.\f') gg'= \psi gg'$, with $\psi=\f(g.\f')$. As above set $\psi_t=\psi|_{X_t}$, and as $gg'=(g_0g'_{\s(0)},g_1 g'_{\s(1)})\s\s'$, the embedding is:
$$\F(\g\g')=(\psi_0 g_0 g'_{\s(0)},\psi_1 g_1 g'_{\s(1)})\s\s'. $$
On the other hand:
\begin{eqnarray*}
\F(\g)\F(\g') &=& (\f_0g_0, \f_1g_1)\s(\f_0'g_0' ,\f_1'g_1')\s' \\
&=& (\f_0g_0 \f_{\s(0)}'g_{\s(0)}',\f_1g_1 \f_{\s(1)}'g_{\s(1)}')\s\s' \\
&=& (\f_0(g_0.\f_{\s(0)}') g_0g_{\s(0)}',\f_1(g_1.\f_{\s(1)}') g_1g_{\s(1)}')\s\s'
\end{eqnarray*}
There remains to check $\psi_t=\f_t(g_t.\f_{\s(t)}')$, and indeed for any $y \in X_t \simeq X$:
\begin{eqnarray*}
\psi_t(y) &=& \psi(ty)=(\f(g.\f'))(ty)=\f(ty)((g.\f')(ty))=\f(ty)\f'(ty.g) \\
&=& \f(ty) \f'(\s(t) (y.g_t))=\f_t(y)\f_{\s(t)}'(y.g_t)=\f_t(y)(g_t.\f_{\s(t)}')(y).
\end{eqnarray*}
\end{proof} 

The embedding $\psi: \G_\o \hookrightarrow \G_{\s\o} \wr S_2$ can also be used at the word level. Let us describe the {\it rewriting process} of a given word of the form $w=a^{i_1}k_1ak_2\dots ak_ra^{i_2}$, for $k_i=\f_{f_i}v_i$ in $\{\f_f v| v \in \{id,b_\o,c_\o,d_\o\},f \in F\}$, which is said {\it pre-reduced}. Note that any reduced representative word in $\G_\o$ has this form.

Any such word can be rewritten $w=k_1^ak_2k_3^ak_4\dots k_ra^{i_3}$ or $w=k_1k_2^a\dots k_r a^{i_4}$, where $i_j \in \{0,1\}$. Note also that $k=\f_f v_\o=(u^v(\o_0),\f_fv_{\s\o})=(u^v(\o_0),k)$ and $k^a=(k,u^v(\o_0))$ and remind $u^v(\o_0) \in \{id,a\}$. This permits to rewrite $w=(w_0,w_1)\s(w)$ via the wreath product embedding, and $w_0,w_1$ appear as products of the type $w_0=a^{\e_1}k_2a^{\e_3}k_4\dots k_r$ and $w_1=k_1a^{\e_2}\dots a^{\e_r}$ for $\e_j \in \{0,1\}$. Now reduce $w_0,w_1$ to obtain pre-reduced words in $S_{\s\o}$, by using the rule $k_ia^0k_{i+1}=k_ik_{i+1}=\f_{f_i}v_i\f_{f_{i+1}}v_{i+1}=\f_{(f_if_{i+1})}(v_iv_{i+1})$. 

The rewritting process associates to $w$ this representation $w=(w_0,w_1)\s(w)$ where $\s(w)$ is the image of $w$ in the quotient group $S_2$ acting at the root.

The process can be iterated, which provides for any level $p$ a representation $w=(w_1,\dots,w_{2^p})\s_p(w)$ with $w_i$ pre-reduced words in $S_{\s^p\o}$ and $\s_p(w) \in Aut(T_2(p))=S_2 \wr \dots \wr S_2$ with $p$ factors describes the action of $w$ on the subtree $T_2(p)$ consisting of the first $p$ levels.

\begin{definition}
Given a pre-reduced word $w$ in $S_\o$, define $T(w)$, called {\it minimal tree} of $w$, to be the minimal regular rooted subtree of $T$ such that for any leaf $z$ in $\partial T(w)$, one has $|w_z|_{pr} \leq 1$ for the word $w_z$ obtained by iterated rewritting process, where $|w|_{pr}$ is the number of factors $k_i=\f_{f_i} v_i$ in a pre-reduced word $w_z$. Remind that a subtree $T$ is rooted if it contains the root and regular if any vertex in $T$ either has its two descendants in $T$ or none of them. Note that the leaves of $\partial T(w)$ have depth at most $\log_2|w|$ because $w_0,w_1$ have length $\leq \frac{|w|+1}{2}$.
\end{definition}

The tree $T(w)$ allows a nice description of the action of a word $w$ in $\G_\o$ on $T$. Indeed, the group element $\g=_{\G_{\o}} w$ is described by the following data. First the minimal tree $T(w)$, secondly the permutations $\s_v \in S_2$ describing the action at vertex $v$ in the interior of $T(w)$ and third the short words $w_z=a^{\e_z}\f_{f_z}v_za^{\d_z}$ for $z \in \partial T(w)$. The latter can be refined in the tree action $a^{\e_z}v_za^{\d_z}$ as an automorphism of $T_z$ the subtree issued from the vertex $z$ and the boundary function $\f(x)=id_F$ for all $x \in \partial T_z \setminus \{z\e_z(1)\r\}$ and $\f(z\e_z(1)\r)=f_z$.

\begin{figure}
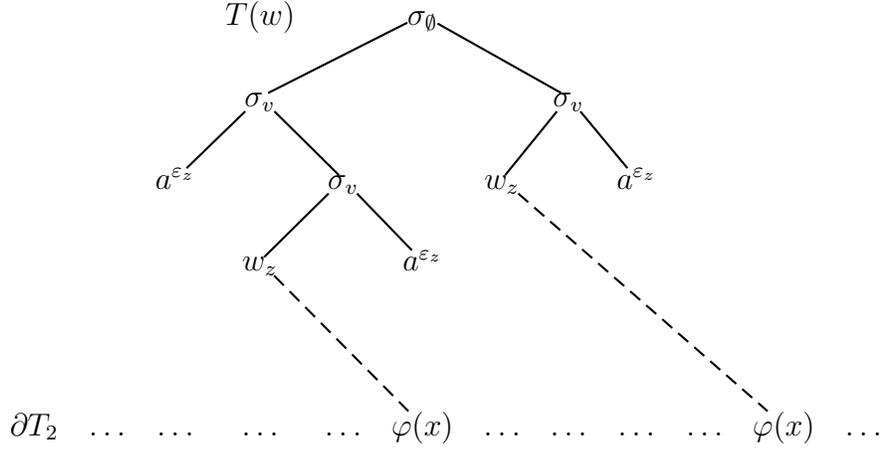

\begin{center}
$\begin{psmatrix}[rowsep=0.6cm,colsep=0.4cm]
& & & T(w) & & \s_\emptyset & & & & & \\
& & & \s_v & & & & \s_v & & & \\
& & a^{\e_z} & & \s_v & & w_z & & a^{\e_z} & & \\ 
& & & w_z & & a^{\e_z} & & & & & \\
& & & & & & & & & & \\
\partial T_2 & \dots & \dots & \dots & \dots & \f(x) & \dots & \dots & \dots & \dots & \f(x) & \dots 
\psset{arrows=-}
\ncline{1,6}{2,4}
\ncline{1,6}{2,8}
\ncline{2,4}{3,3}
\ncline{2,4}{3,5}
\ncline{2,8}{3,7}
\ncline{2,8}{3,9}
\ncline{3,5}{4,4}
\ncline{3,5}{4,6}
\ncline[linestyle=dashed]{4,4}{6,6}
\ncline[linestyle=dashed]{3,7}{6,11}
\end{psmatrix}$
\end{center}
\caption{\label{fig1} Description of the action of a word $w$ via the minimal tree $T(w)$}
\end{figure}

Call $z \in \partial T(w)$ an {\it active leaf} if $|w_z|_{pr}=1$, an {\it inactive leaf} if $|w_z|_{pr}=0$, denote $S(w)$ the set of active leaves of $w$, and $s(w)=\#S(w)$ its size. Mind that if $z$ is inactive then $w_z=a^{\e_z} \in S_2$ is just a permutation. Note also that regarding the rules of rewritting process $|\f_{id_F}|_{pr}=1$ so that an active leaf does not necessarily act on the tree, nor its boundary (see figure \ref{fig1}).

However, it appears from the description above that the support of $\f:X \rightarrow F$ associated to $w$ is included in $\{z\e_z(1)\r|z \textrm{ is an active leaf}\}$. Call this set the {\it support a priori} of $\f$, denoted $supp^{ap}(\f)$. Note that for $w=a^{i_1}\f_{f_1}v_1a\dots a\f_{f_r}v_ra^{i_2}$, if $i_1,i_2, v_j$ are kept fixed and $(f_1,\dots,f_r)$ are taking all possible values, then any function with support included in $supp^{ap}(\f)$ can be obtained. In particular, the support a priori of the function $\f$ for the word $w$ depends only on the image in the quotient $\G_\o \rightarrow G_\o$, $w \mapsto g=a^{i_1}v_1a\dots av_ra^{i_2}$.

\begin{remark}\label{foret}
In order to clarify the notion of support a priori, let us introduce a notion associated to the word combinatorics of the rewritting process of a fixed word $w$. For $z$ an active leaf of $T(w)$, the rewritting process provides $f_z$ as a product of terms $f_i^{z'}$ in $w_{z'}$ (where $z'$ is the first ascendant of $z$), which are themselves products of terms $f_j^{z''}$ in $w_{z''}$, etc. so eventually $f_z$ is a product of terms $(f_j)_{j \in J(z)}$ for a subset $J(z) \subset \{1,\dots,r\}$. Note that in this situation: $\bigsqcup_{z \in S(w)} J(z) =\{1,\dots,r\}$.

More generally, if $y$ is a vertex of $T$, the rewritting process of $w$ provides $w_y=a^{i_1^y}\f_{f_1^y}v_1^ya\dots a\f_{f_{r^y}^y}v_{^y}^yra^{i_2^y}$, and each factor $f_i^y$ is obtained as an ordered product:
\begin{eqnarray} \label{prod} f_i^y=\prod_{j \in I(y,i)}f_j^{y'}, \end{eqnarray}
where $y'$ is the first ascendant of $y$, and $\sqcup I(y,i) = \{1,\dots,r_{y'}\}$ where the disjoint union runs over all direct descendants $y$ of $y'$ and $i \in \{1,\dots, r_y\}$. 

Now the graph with vertex set $(f_i^y)_{y \in T, i \in \{1,\dots,r_y\}}$ and edges pairs of elements appearing on different sides of all possible products (\ref{prod}) is a forest, called the {\it ascendance forest} of $w$. It describes the combinatorics of the rewritting process of the word $w$. It depends only on $g=a^{i_1}v_1a\dots av_ra^{i_2}$. Precisely, this graph is a finite union of trees rooted in $f_z$ for each $z \in S(w)$ and with respective sets of leaves $\{f_j|j \in J(z)\}$. The ordered product $f_z=\prod_{j \in J(z)} f_j$ shows that indeed, the function $\f$ can take any value at the point $z\e_z(1)\r$.
\end{remark}

\begin{proposition}(Activity of a pre-reduced word)\label{activite}
The activity $s(w)$ of a pre-reduced word $w=a^{i_1}\f_{f_1}v_1a\dots a\f_{f_r}v_ra^{i_2}$ in $\G_\o,S_\o$, which counts equivalently
\begin{enumerate}
\item the size of the set $S(w)$ of active leaves in the minimal tree $T(w)$,
\item the number of components (i.e. trees) in the ascendance forest of $w$,
\item the size of the support a priori $supp^{ap}(\f)$,
\item the size of the inverted orbit $\Oc(g^{-1})$ of the word $g^{-1}$ in the sense of \cite{BE},
\end{enumerate}
depends only on the word $\underline{w}=a^{i_1}v_1a\dots av_ra^{i_2}$ in $G_\o$ and satisfies under rewritting process $w=(w_0,w_1)\s(w)$, with $w_0,w_1$ in $S_{\s\o}$:
$$s(w)=s(w_0)+s(w_1). $$
Also there exists a constant $C$ depending only on $\#F$ such that:
$$\#\{\g \in \G_\o| \exists w=_{\G_\o}\g,s(w) \leq s\} \leq C^s.$$
\end{proposition}

\begin{proof}
The equivalence of (1), (2) and (3), as well as the behavior of activity function under rewritting process follow from the descriptions above. Proceed by induction on $r$ to show equivalence with (4). If $w=a^{i_1}\f_{f_1}v_1a\dots a\f_{f_r}v_ra^{i_2}=_{F\wr \G}\f g$ then $w\f_{f_n}v_n=\f (g.\f_{f_n})gv_n$. The point $g^{-1}(1^\infty)$ is added to the support a priori of $\f$. This shows $supp^{ap}(\f)=\{(a^{i_1}v_1\dots v_ka^{i_2})^{-1}(1^\infty)|k \leq n\}=\Oc(g^{-1})$. Mind that the inverse appears as a difference with \cite{BE} notations, replacing $g f$ by $\f g$ for elements of $F \wr G$. Then $\f g =(g.f)g $ and $g^{-1}supp (f)=supp(\f)$.

There remains only to show that the number of elements described grows at most exponentially fast with $s(w)$. First check that $2s(w) \geq \#\partial T(w)$ when $s(w) \geq 1$, by induction on $s(w)$. If $|w|_{pr}=1$, then $T(w)$ is just the root of $T$. Now if $s(w) \geq 2$, then $s(w_0),s(w_1) \geq 1$ by pre-reduction of $w$, so that induction ensures $2s(w_t) \geq \#\partial T(w_t)$, and the result follows from $\#\partial T(w_0)+\#\partial T(w_1) \geq \#\partial T(w)$ by construction of minimal trees.
Now if $s(w) \leq s$, the minimal tree $T(w)$ has size $\leq 2s$. There is $4^{2s}$ possibilities for $T(w)$ (Catalan numbers), and then $2^{\#\textrm{interior}(T(w))}\leq 2^{2s}$ choices for the interior permutations $\s_v$ for interior vertices $v$ and finally $(2^2.4.\#F)^{\#\partial T(w)}\leq C^{2s}$ choices for the boundary short words $a^{\e_z}\f_{f_z}v_za^{\d_z}$.
\end{proof}

\begin{corollary}\label{actgro} The relation between word activity and growth function is two-fold:
\begin{enumerate}
\item $b_{\G_\o}(r) \geq \#F^{s(w)}$ for any $|w| \leq r$.
\item $b_{\G_\o}(r) \leq C^{\max\{s(w)|r \geq |w|\}}.$
\end{enumerate}
\end{corollary}

In particular, word activity governs the growth function.

\begin{proof}
Point (1) is clear from remark \ref{foret} and point (2) from proposition \ref{activite}.
\end{proof}

\section{Technics of estimation}

\subsection{Growth Lemma} The following lemma is used to estimate upper bounds on the growth of activity hence on the growth of groups. It improves on previous versions such as the Growth Theorem in \cite{MP} and Lemma 4.3 in \cite{BE} by keeping track of the constants in terms of the bound on the sequence $p(r)$ of variable depth of recursion.

\begin{lemma}\label{growthlemma}
Given $\eta$ and a parameter $\l \in [0,1]$, set $\a=\frac{\log(2)}{\log(2)-\l\log(\eta)}$, so that $\a$ satisfies $2=\left( \frac{2}{\eta^\l}\right)^\a$.

Let $\D:\N \rightarrow \N$ be a function such that for any $r$ there exists $q(r)\leq p(r)$ and $l_1,\dots,l_{2^{p(r)}}$ integers such that, for a constant $C$:
\begin{enumerate}
\item $l_1+\dots +l_{2^{p(r)}} \leq \eta^{q(r)}r+2^{p}C$,
\item $\D(r) \leq \D(l_1)+\dots+\D(l_{2^{p(r)}})$, 
\item $\frac{q(r)}{p(r)}\geq \l$.
\end{enumerate}
Suppose moreover that $p(r)\leq P$ is bounded. Then $\D(r) \leq L r^\a$ for some constant $L=L(C,P)$.

Assume given a trivial bound $\D(r) \leq Kr$. Then $L$ can be chosen $L=A^P$ for $A$ depending only on $C$ and $K$.
\end{lemma}

\begin{proof}
Choose $R_0$ big, to be determined later. Choose $R \geq R_0$ large enough so that the function:
$$ \D^{\ast \ast}(r)= \left\{ \begin{array}{ll} \left(r-r^{\frac{1}{2}} \right)^\a & \textrm{if $r\geq R$,} \\ 1+\frac{r}{R}(\D^{\ast \ast}(R)-1) &
\textrm{if $r\leq R$,} \\ 
\end{array} \right.$$ 
is concave (it is also non decreasing). Choose $M$ large enough so that for all $r\geq M$, $\frac{\eta^{q(r)}}{2^{p(r)}}r\geq R$, and let $L\geq 1$ be large enough so that $\D^\ast (r) = L \D^{\ast \ast}(r) \geq \D(r)$ for all $r \leq M$. Let $r >M$, there exists $p(r),q(r),l_i$, with $\D(l_i)\leq \D^\ast(l_i)$ by induction, and using successively (2), induction, concavity of $\D^\ast$, (1) and the choice of $M$:
\begin{eqnarray*}
\D(r)&\leq& \D(l_1)+\dots+\D(l_{2^p}) \\
&\leq& \D^\ast(l_1)+\dots+\D^\ast(l_{2^p}) \\
&\leq& 2^p\D^\ast\left(\frac{1}{2^p}(l_1+\dots+l_{2^p}) \right) \\
&\leq& 2^p\D^\ast\left(\frac{1}{2^p}(\eta^qr+2^pC) \right) \\
&=& 2^p L\left( \left(\frac{\eta^{\frac{q}{p}}}{2} \right)^p r + C - \left( \frac{\eta^q}{2^p}r+C \right)^{\frac{1}{2}} \right)^\a, \\
&=& L\left( \left(2^{\frac{1}{\a}}\frac{\eta^{\frac{q}{p}}}{2} \right)^p r + 2^{\frac{p}{\a}}C - 2^{\frac{p}{\a}} \left( \frac{\eta^q}{2^p}r+C \right)^{\frac{1}{2}} \right)^\a, \\
\end{eqnarray*}
Now $\frac{q}{p} \geq \l$ ensures $\left(2^{\frac{1}{\a}}\frac{\eta^{\frac{q}{p}}}{2} \right)^p \leq1$, so:
\begin{eqnarray*}
\D(r)&\leq& L \left(r+ 2^{\frac{p}{\a}}C - 2^{\frac{p}{\a}} \left( \frac{\eta^q}{2^p}r+C \right)^{\frac{1}{2}}\right)^\a \\
&\leq& L\left(r-r^{\frac{1}{2}} \right)^\a=\D^\ast(r).
\end{eqnarray*}
The last inequality holds when $r$ is big enough so that:
$$2^{\frac{p}{\a}}\left(\frac{\eta^q}{2^p}r+C \right)^{\frac{1}{2}}-2^{\frac{p}{\a}}C \geq r^{\frac{1}{2}}.$$
Observe that $\frac{2}{\eta^\l}\left(\frac{\eta^\frac{q}{p}}{2} \right)^{\frac{1}{2}}\geq \sqrt{2\eta}>1 $ so the latter is true when:
$$ (\sqrt{2\eta})^{\frac{p}{2}}r^{\frac{1}{2}}  \geq r^{\frac{1}{2}} + 2^{\frac{p}{\a}}C,$$
which holds when $r \geq a_0^P=R_0=\frac{2^{\frac{2p}{\a}}C^2}{(\sqrt{2\eta}^{\frac{p}{2}}-1)^2}$ with a constant $a_0$ depending only on $C$. For $P$ big, $R=R_0$ and so $M=\left( \frac{2}{\eta} \right)^P R_0=\left( \frac{2}{\eta} a_0\right)^P$. It is sufficient to take $L\D^{\ast\ast}(M) \geq KM$ so $L\geq K\left(\frac{2}{\eta}a_0 \right)^P$.
\end{proof}

\subsection{Localization}
The asymptotic behavior of the growth of $\G_\o$ depends on the asymptotic of $\o$. On the other hand, the description of a ball of a given radius in $\G_\o$ requires only some first terms of $\o$. The following lemma of localization is helpful to study growth of groups $\G_\o$ for non periodic sequences $\o$. 

\begin{lemma}\label{localization}
Suppose that the sequence $\o$ is not asymptotically constant, then the ball $B_{\G_\o}(r)$ of radius $r$ for the word norm with respect to the generating set $S_\o=\{a\} \sqcup \{\f_fv|v \in \{id_{G_\o},b_\o,c_\o,d_\o\},f \in F\}$ depends only on $\o_0\o_1\dots\o_k$ for $k=\log_2(r)$.
\end{lemma}

The Cayley graph $Cay(\G,S)$ of a group $\G$ with generating set $S$ is the colored graph with vertices $\g$ in $\G$ and edges $(\g,\g s)$ of color $s$ in $S$. The ball $B_\G(r)$ of radius $r$ is the subgraph obtained by restriction to vertices and ends of edges such that $|\g| \leq r$ for the word norm for $S$.

\begin{proof}
The ball $B_\o(1)$ of $\G_\o$ for the generating set $S_\o$ is independent of $\o$ among sequences that are not constant, it consists of the Cayley graph $Cay(F \times V, F \times V)$ together with an edge from the neutral element leading to the vertex $a$. By proposition \ref{CE}, the ball $B_{\G_\o}(r)$ can be described using $B_{\G_{\s\o}}(\frac{r+1}{2})$ and the wreath product recursion (\ref{generateurs}), i.e. $\o_0$. Indeed, an element $\g$ admits a reduced representative word $w=a^{i_1}k_1ak_2\dots ak_{r/2}a^{i_2}$ and so $\g=(\g_0,\g_1)\e^s$ with $|\g_0|,|\g_1| \leq \frac{r+1}{2}$ by rewritting process. By iteration, $B_{\G_\o}(r)$ is described by $B_{\G_{\s^k\o}}(\frac{r}{2^k}+1)$ and $\o_0\dots\o_k$.
\end{proof}

\begin{remark}
When $\o=0^\infty$ is constant, the generator $d_\o$ acts trivially on the rooted tree $T$, hence is identity, so that the Klein group $V$ degenerates into a group $S_2$, and $G_\o=\langle a,b_{0^\infty}|a^2=b^2=id \rangle=D_\infty$ is dihedral infinite. However, the whole sequence $\o$ is required to obtain this information. The group $\tilde{G}_{0^\infty}$ obtained by ``finite information'' (concretely as a limit group of $G_{0^k(012)^\infty}$ for instance) is in fact the group $\tilde{G}_{0^\infty} \simeq S_2 \wr_X G_{0^\infty}=\langle d_{0^\infty} \rangle \wr_X \langle a, b_{0^\infty} \rangle$, which is metabelien of exponential growth. Indeed, the wreath product embeddings $G_{0^k(012)^\infty} \hookrightarrow G_{0^{k-1}(012)^\infty} \wr S_2$, $b_k=(a,b_{k-1}$, $d_k=(1,d_{k-1})$ and $b_\infty=(a,b_\infty)$, $d_\infty=(1,d_\infty)$ coincide on the $k$ first levels. As the balls of radius 1 of the groups $G_{0^k(012)^\infty}$ and $\langle d_{0^\infty} \rangle \wr_X \langle a, b_{0^\infty} \rangle$ also coincide, the argument of the proof of lemma \ref{localization} shows that these groups coincide on their balls of radius $2^k$, showing isomorphism $\tilde{G}_{0^\infty} \simeq S_2 \wr_X G_{0^\infty}$. This group played a crucial role in the construction of antichains of growth functions, cf. section 6 in \cite{Gri1}. 
\end{remark}

\subsection{Asymptotic growth} Opposed to localization, the asymptotic behavior of the growth depends only on the asymptotic of $\o$.

\begin{proposition}\label{comparaisonfinie}
For generating sets $S_\o=\{a\} \sqcup \{\f_fv| v \in \{id,b_\o,c_\o,d_\o\},f \in F\}$, the growth function of $\G_\o=F\wr_X G_\o$ satisfies for all $r$:
$$ b_{\G_{\s\o}}\left(\frac{r-1}{2}\right) \leq b_{\G_\o}(r) \leq 2b_{\G_{\s\o}}\left(\frac{r+1}{2}\right)^2. $$
Also by iteration:
$$ b_{\G_{\s^k\o}}\left(\frac{r}{2^k}-1\right) \leq b_{\G_\o}(r) \leq 2^{2^{k}}b_{\G_{\s^k\o}}\left(\frac{r}{2^k}+1\right)^{2^k}. $$
\end{proposition}

\begin{proof}
Let $\g=\f g$ belong to $B_{\G_\o}(r)$. It admits a minimal representative word $w=_{\G_\o}\g$ of length $r$, which is uniquely described after rewritting process as $w=(w_0,w_1)\s(w)$ with $|w_0|,|w_1| \leq \frac{r+1}{2}$. Conclude that $\g$ is determined by two elements $\g_0,\g_1$ in $B_{\G_{\s\o}}(\frac{r+1}{2})$ and a permutation $\s(w)$ in $S_2$, which proves the upper bound.

Suppose $\o_0 \neq 1$ and let $\g_0$ belong to $B_{\G_{\s\o}}(\frac{r-1}{2})$. It admits a minimal representative word $w_0=a^{i_1}k_1ak_2a\dots ak_la^{i_2}$ of length $\leq \frac{r-1}{2}$. Set $w=b_\o \bar{k}_1^a b_\o \bar{k}_2^a \dots b_\o \bar{k}_l^a b_\o^{i_2}$ if $i_1=1$ and $w= \bar{k}_1 b_\o^a \bar{k}_2 b_\o^a \dots b_\o^a \bar{k}_l b_\o^{i_2a}$ if $i_1=0$ of length $\leq r$, where $\bar{k}_j=\f_{f_j}v^j_\o$ for $k_j=\f_{f_j}v^j_{\s\o}$. Proposition \ref{CE} and relations (\ref{generateurs}) from section \ref{definition} guarantee that $w=(w_0,w_1)\s(w)$ for some $w_1,\s(w)$. Now if $w=_{\G_\o}w'$, then $w_0=_{\G_{\s\o}}w_0'$, so that $B_{\G_{\s\o}}(\frac{r-1}{2})$ injects into $B_{\G_\o}(r)$. (Note that when $\o_0=1$, the same computation works if $b_\o$ is replaced by $d_\o$.)
\end{proof}

\section{Activity and growth}

\subsection{Activity of some words and lower bound on growth} Denote $||A||$ the operator norm of a $3\times 3$ matrix $A$ acting on $\C^3$. The proposition 4.7 in \cite{BE} generalizes as:
\begin{proposition} \label{somecontraction} Denote:
$$A_0=\left( \begin{array}{ccc} 2 & 0 & 1 \\ 0 & 2 & 1 \\ 0 & 0 & 1 \\ \end{array} \right), A_1=\left( \begin{array}{ccc} 2 & 1 & 0 \\ 0 & 1 & 0 \\ 0 & 1 & 2 \\ \end{array} \right),A_2=\left( \begin{array}{ccc} 1 & 0 & 0 \\ 1 & 2 & 0 \\ 1 & 0 & 2 \\ \end{array} \right). $$
There is $C>0$ such that for any $k$, there is a word $w_k$ in $\G_\o,S_\o$ such that $s(w_k) \geq 2^k$ and $|w_k| \leq C||A_{\o_0}\dots A_{\o_k}||$.
\end{proposition}
\begin{proof}
Consider the subsemigroup $\O_\o'=\{ab_\o,ac_\o,ad_\o\}^\ast \subset \O_\o$, and define the pull back substitution $\z: \O'_{\s\o} \rightarrow \O'_\o$ by:
$$\begin{array}{llll}
\z(ab_{\s \o})=ab_\o ab_\o & \z(ac_{\s \o})=ac_\o ac_\o & \z(ad_{\s \o})=ab_\o ad_\o ac_\o & \textrm{ if $\o_0=0$,} \\
\z(ab_{\s \o})=ab_\o ab_\o & \z(ac_{\s \o})=ad_\o ac_\o ab_\o & \z(ad_{\s \o})=ad_\o ad_\o & \textrm{ if $\o_0=1$,} \\
\z(ab_{\s \o})=ac_\o ab_\o ad_\o & \z(ac_{\s \o})=ac_\o ac_\o & \z(ad_{\s \o})=ad_\o ad_\o & \textrm{ if $\o_0=2$.} \\
\end{array}$$
Such a pull back substitution is designed so that $\z(au)=(ua,au)$ when $au$ is apre-reduced word containing an even number of $v$'s (where $v=d$ if $\o_0=0$, $v=c$ if $\o_0=1$ and $v=b$ if $\o_0=2$). Indeed, the following relations hold (take $\o_0=0$, similar otherwise):
$$\begin{array}{ll} \z(ab)=abab=(ba,ab), & baba=(ab,ba), \\
\z(ac)=acac=(ca,ac), & caca=(ac,ca), \\
\z(ad)=abadac=(d,ada)a, & badaca=(ada,d)a. \end{array} $$
The pull back of $v_{\s\o}$ furnishes $v_\o$ on both components of the wreath product. The $a$'s behave conveniently under the parity condition. 

Given a word $w_0=au_0$ in $\O'_{\s^k \o}$, define by induction $\z(au_{k-1})=au_k=w_k \in \O'_\o$. The initial word $u_0$ can be chosen among the generators $\{b_{\s^k\o},c_{\s^k\o},d_{\s^k\o}\}$ so that $\z(au_0)=avav$ for another generator $v$ of $G_{\s^{k-1}\o}$, so that $\z(au_k)$ always has an even number of $v$'s, and the inverted orbit of $au_k$ can be studied by induction via:
$$\z(au_{k-1})=au_k=(u_{k-1}a,au_{k-1}) \textrm{ and } u_ka=(au_{k-1},u_{k-1}a). $$
Proposition \ref{activite} now ensures that:
$$s(au_k) \geq s(au_{k-1})+s(u_{k-1}a) \textrm{ and } s(u_ka) \geq s(u_{k-1}a)+s(au_{k-1}),$$
which is integrated in $s(au_k)\geq 2^k$. 

To estimate the length of $w_k=\z(w_{k-1})$, it is sufficient to count the numbers $|w|_{b_\o},|w|_{c_\o},|w|_{d_\o}$ of generators $b_\o,c_\o,d_\o$ appearing in $w$, since the total length is controlled by $|w| \leq 2(|w|_{b_\o}+|w|_{c_\o}+|w|_{d_\o})$. The construction of the pull back substitution $\z$ provides the relations:
$$A_{\o_0}\left(\begin{array}{c} |w_{k-1}|_{b_{\s\o}} \\ |w_{k-1}|_{c_{\s\o}} \\ |w_{k-1}|_{d_{\s\o}} \\ \end{array} \right)=\left(\begin{array}{c} |w_k|_{b_\o} \\ |w_k|_{c_\o} \\ |w_k|_{d_\o} \\ \end{array} \right), $$
for the matrices $A_0,A_1,A_2$, and eventually by induction:
$$A_{\o_0}A_{\o_1} \dots A_{\o_k} \left(\begin{array}{c} |w_0|_{b_{\s^k\o}} \\ |w_0|_{c_{\s^k\o}} \\ |w_0|_{d_{\s^k\o}} \\ \end{array} \right)= \left(\begin{array}{c} |w_k|_{b_\o} \\ |w_k|_{c_\o} \\ |w_k|_{d_\o} \\ \end{array} \right), $$
so that $|w_k| \leq C||A_{\o_0}\dots A_{\o_k}||$. 
\end{proof} 

The matrices $A_0,A_1,A_2$ are cyclic conjugates $A_1=CA_0C^{-1}$ and $A_2=C^{-1}A_0C=CA_1C^{-1}$, so that $A_0^{k_1}A_1^{k_2}A_2^{k_3}\dots=A_0^{k_1}CA_0^{k_2}CA_0^{k_3}\dots$ with
$$C=\left( \begin{array}{ccc} 0 & 1 & 0 \\ 0 & 0 & 1 \\ 1 & 0 & 0 \\ \end{array} \right), A_0C=\left( \begin{array}{ccc} 1 & 2 & 0 \\ 1 & 0 & 2 \\ 1 & 0 & 0 \\ \end{array} \right).$$
The matrix $A_0C$ has caracteristic polynomial $X^3-X^2-2X-4$ with positive real root $\frac{2}{\eta}$, and two complex conjugate roots of smaller absolute value, hence spectral radius $\rho(A_0C)=\frac{2}{\eta}$. (Remind that $\eta$ is the positive root of $X^3+X^2+X-2$.)

\begin{examples}\label{exemplesperiodiques}
\begin{enumerate}
\item For $\o=(012)^\infty$, the spectral radius theorem gives: 
$$||A_{\o_0}\dots A_{\o_k}|| \leq ||(A_0C)^{k+1}|| \leq C \r(A_0C)^k=C \left(\frac{2}{\eta} \right)^k. $$
\item For other periodic sequences, similar bounds are obtained, as for instance $\o=(001122)^\infty$, then:
$$ ||A_{\o_0}\dots A_{\o_k}|| \leq ||(A_0^2C)^{\frac{k+1}{2}}|| \leq C \r(A_0^2C)^{\frac{k}{2}}, $$
where the spectral radius $\r(A_0^2) \approx 5.63$ is the positive root of $X^3-3X^2-12X-16$.
\end{enumerate}
\end{examples}

Such estimates for periodic sequences are not usually sharp enough. The following lemma is the key step to obtain the lower bound on the growth function in theorem \ref{precisegrowth}:
\begin{lemma}\label{spectralradii}
Let $\o=0^{m_1}(012)^{n_1}0^{m_2}(012)^{n_2}\dots$, with $m_i,n_i \rightarrow \infty$. There exists a constant $C$, such that for every $\e>0$ and $k=\sum_{i=1}^j m_i+3n_i$ big enough:
$$||A_{\o_0}\dots A_{\o_k}|| \leq C^{\e k}\r(A_0)^{\sum m_i} \r(A_0C)^{3\sum n_i}=C^{\e k} 2^{\sum m_i} \left( \frac{2}{\eta} \right)^{3\sum n_i}. $$
\end{lemma}

Note that this inequality is essentially optimal.

\begin{proof}
By the spectral radius theorem, there exists $C$ such that $||A_0^m|| \leq C \r(A_0)^m$ and $||(A_0C)^{3n}|| \leq C \r(A_0C)^{3n}$, so:
\begin{eqnarray*}
||A_{\o_0}\dots A_{\o_k}|| &\leq& ||A_0^{m_1}(A_0C)^{3n_1}A_0^{m_2}(A_0C)^{3n_2}\dots|| \\
&\leq& ||A_0^{m_1}||.||(A_0C)^{3n_1}||.||A_0^{m_2}||.||(A_0C)^{3n_2}||\dots \\
&\leq& C^j \r(A_0)^{\sum_{i=1}^jm_i}\r(A_0C)^{3\sum_{i=1}^jn_i}, 
\end{eqnarray*}
where $j=o(k)$ since $m_i,n_i \rightarrow \infty$.
\end{proof}
Note that if $m_i,n_i$ are of the order $\log i$, then $j \approx \frac{k}{\log k}$, and if $m_i,n_i$ are of the order $i^\theta$, then $j \approx k^{\frac{1}{\theta+1}}$.

\subsection{Activity of all words and upper bound on growth}\label{allactivity} Say a sequence $\o=\o_0\o_1\o_2\dots$ in $\{0,1,2\}^\N$ is {\it rotating} if $\o_{i+1} \in \{\o_i, \o_i+1\} \textrm{ mod $3$}$ for all $i$. Remind that $\eta$ is the positive root of $X^3+X^2+X-2$. Adapting \cite{Bar1} to rotating sequences $\o$, define a length on $G_\o$ by assigning weights to the generating set $\langle a, b_\o,c_\o,d_\o \rangle$. Set $||a||=1-\eta^3$ and:
$$\begin{array}{llll}
 \textrm{ if $\o_0=0$,} & ||b_\o||=\eta^3, & ||c_\o||=1-\eta^2, & ||d_\o||=1-\eta,  \\
 \textrm{ if $\o_0=1$,} & ||b_\o||=1-\eta^2, & ||c_\o||=1-\eta, & ||d_\o||=\eta^3, \\
 \textrm{ if $\o_0=2$,} & ||b_\o||=1-\eta, & ||c_\o||=\eta^3, & ||d_\o||=1-\eta^2. \\
\end{array}$$
This defines a length on $G_\o$ for which the minimal representative words are pre-reduced ($\eta$ is chosen so that this is the case, see lemma 4.1 in \cite{Bar1}), which is obviously equivalent to the usual word length $\frac{1}{C}|w| \leq ||w|| \leq C|w|$, and designed so that if $\o_1=\o_0+1$, then:
\begin{eqnarray*}
\e_b(\o_0)||a||+||b_{\s \o}|| &=& \eta(||a||+||b_\o||), \\
\e_c(\o_0)||a||+||c_{\s \o}|| &=& \eta(||a||+||c_\o||), \\
\e_d(\o_0)||a||+||d_{\s \o}|| &=& \eta(||a||+||d_\o||), \\
\end{eqnarray*}
where: $\e_b\left( \begin{array}{c} 0 \\ 1 \\ 2 \\ \end{array}\right)=\left( \begin{array}{c} 1 \\ 1 \\ 0 \\ \end{array}\right),\e_c\left( \begin{array}{c} 0 \\ 1 \\ 2 \\ \end{array}\right)=\left( \begin{array}{c} 1 \\ 0 \\ 1 \\ \end{array}\right),\e_d\left( \begin{array}{c} 0 \\ 1 \\ 2 \\ \end{array}\right)=\left( \begin{array}{c} 0 \\ 1 \\ 1 \\ \end{array}\right), $
and if $\o_0=\o_1$, then the factor $\eta$ on the right-hand sides disappears. 

This length on $G_\o$ can be extended to a function on the set of words in the generating set $\langle S_\o \rangle$ of $\G_\o$ by $||\f_fv||=||v||$ if $v \in \{b_\o,c_\o,d_\o\}$ and $||\f_f id||=0$. Note that even though $||w||$ is not a length on $\G_\o$ it is still equivalent to the length of pre-reduced words, i.e. $\frac{1}{C}|w| \leq ||w|| \leq C |w|$, because if $w=a^{i_1}\f_{f_1}v_1a\f_{f_2}v_2\dots a\f_{f_r}v_ra^{i_2}$, then $||w||=||\underline{w}||$ for $\underline{w}=a^{i_1}v_1av_2\dots av_ra^{i_2}$, which is bilipschitz equivalent to $r$.

The following statement generalizes Lemma 4.2 in \cite{BE}. 
\begin{lemma}\label{contraction}
Let $w$ be a pre-reduced word of $\G_\o,S_\o$ with rewritting process giving $w=(w_0,w_1)\e^s$, then: 
$$||w_0||+||w_1|| \leq \eta^{q(\o_0,\o_1)} ||w||+C,$$
where $C=\eta ||a||$, $q(\o_0,\o_1)=0$ if $\o_1=\o_0$, $q(\o_0,\o_1)=1$ if $\o_1=\o_0+1$ and the left-handside lengths are in $\G_{\s\o}$, the right-hand side one in~$\G_\o$.
\end{lemma}

\begin{proof}
The inequalities for $\underline{w}=(\underline{w}_0,\underline{w}_1)\s(w)$ in $G_\o$ and $G_{\s\o}$ are obvious by construction of the length $||.||$, i.e. by choice of $\eta$. They still apply to pre-reduced words in $\G_\o$ and $\G_{\s\o}$.
\end{proof}

In order to estimate the growth function from above, the word activity function 
$$s_\o(r)=\max\{s(w)|w \in (\G_\o,S_\o), |w|\leq r\}$$ 
will be usefull. However, it is smoother to estimate first the bilipschitz equivalent auxiliary
$$\D_\o(r)=\max\{s(w)|w\textrm{ is pre-reduced}, ||w||\leq r\}.$$

\begin{fact}\label{fait}
For any $r$, there exists $l_0,l_1$ integers such that:
\begin{enumerate}
\item $\D_\o(r) \leq \D_{\s\o}(l_0)+\D_{\s\o}(l_1)$, and
\item $l_0+l_1 \leq \eta^{q(\o_0,\o_1)}r+C$.
\end{enumerate}
By induction, there exists $l_1,\dots,l_{2^p}$ integers such that:
\begin{enumerate}
\item $\D_\o(r) \leq \D_{\s^p\o}(l_1)+\dots+\D_{\s^p\o}(l_{2^p})$, and
\item $l_1+\dots +l_{2^p} \leq \eta^{q(\o_0,\dots,\o_p)}r+2^{p+1}C$, where $q(\o_0,\dots,\o_p)$ is the number of $i$ such that $\o_{i+1}=\o_i+1 \mod 2$.
\end{enumerate}
\end{fact}

\begin{proof}
The maximum is realized for a certain word $w$, for which the rewritting process furnishes $w=(w_0,w_1)\s(w)$ with $l_0=||w_0||$ and $l_1=||w_1||$ such that $l_0+l_1 \leq \eta^{q(\o_0,\o_1)}||w||+C$ by lemma \ref{contraction}. Thus:
$$\D_\o(r) =s(w)=s(w_0)+s(w_1) \leq \D_{\s\o}(l_0)+\D_{\s\o}(l_1). $$
\end{proof}

\begin{proposition}\label{corupperbound}
Suppose $\o$ is such that for all $i$, there exists $p(i) \leq P$ such that $q(\o_i,\dots,\o_{i+p(i)})=q(i)$ and $\frac{q(i)}{p(i)} \geq \l$, then:
$$\log b_{\G_\o}(r) \leq A^Pr^\a, \textrm{ for } \a=\frac{\log(2)}{\log(2)-\l\log(\eta)}. $$
\end{proposition}
In particular, if $\o$ is $p$-periodic and $q(\o_0,\dots,\o_p)=q$, then $\log b_{\G_\o}(r) \leq Lr^\a$ for $\a=\frac{\log(2)}{\log(2)-\frac{q}{p}\log(\eta)}.$

\begin{proof}
Set $\D(r)=\sup \{\D_{\s^p\o}(r)|p \in \N\}$. Fact \ref{fait} provides the existence of $l_1+\dots+l_{2^p} \leq \eta^{q(\o_0,\dots,\o_p)}r+2^pC$ such that $\D(r) \leq \D(l_1)+\dots+\D(l_{2^p})$, and so by lemma \ref{growthlemma} there is a constant $A$ such that $\D(r) \leq A^P r^\a$, so that $s_\o(r) \leq A^Pr^\a$ (mind that there is a trivial bound $\D(r) \leq Kr$ because the activity is bounded by the word length). Now corollary \ref{actgro} shows $\log b_{\G_\o}(r) \leq A^Pr^\a$.
\end{proof}

\section{Precise growth estimates}
The particular case of theorem \ref{mainthm} can now be derived. Recall that $\a_0$ is such that $2=\left( \frac{2}{\eta} \right)^{\a_0}$.

\begin{theorem}\label{precisegrowth}
For any $\a \in [\a_0,1]$, there exists a sequence $\o(\a)$ such that $\a(\G_{\o(\a)})=\a$, i.e. 
$$\lim \frac{\log \log b_{\G_{\o(\a)}}(r)}{\log r} = \a. $$
\end{theorem}

\begin{proof} Given $\a$, take $\l$ in $[0,1]$ such that $2=(\frac{2}{\eta^\l})^\a$. Consider a sequence of the form $\o=0^{m_1}(012)^{n_1}0^{m_2}(012)^{n_2}\dots$. Denote the $i$th period $p_i=m_i+3n_i$ and $q_i=3n_i$ the number of steps of rotation of $\o$, and assume both tend to infinity. Suppose moreover that $\frac{q_i}{p_i} \geq \l$ for each $i$ and $\frac{q_i}{p_i} \rightarrow \l$, so that $\sum_{i=1}^j p_i=k_j$ and $\sum_{i=1}^j q_i=\l k_j+o(k_j)$. 

Lemma \ref{localization} of localization allows to use proposition \ref{corupperbound} for $P$ depending on the scale $r$. Indeed, $b_{\G_\o}(r)$ depends only on $\o_0,\dots,\o_k$  for $k=\log_2(r)$, for which $p_i \leq P(k)$, so that:
$$\log b_{\G_\o}(r) \leq A^{P(k)}r^\a \leq r^{\a+\e}, $$
as soon as $P(k) \leq \e \log_A(r)$. In particular, if $\o$ is chosen such that $P(k)=o(\log(r))=o(k)$, the required upper bound holds: $\overline{\a}(\G_\o) \leq \a$.

Concerning lower estimates, the word $w_{k_j}$ introduced in proposition \ref{somecontraction} has length bounded by (lemma \ref{spectralradii}):
$$|w_{k_j}| \leq C^{\e'k_j}\frac{2^{\sum_{i=1}^j p_i}}{\eta^{\sum_{i=1}^j q_i}} \leq C^{2\e'k_j}\left( \frac{2}{\eta^\l}\right)^{k_j} \leq \left( \frac{2}{\eta^\l} \right)^{k_j(1+\e)}.$$
Now lemma \ref{actgro} ensures, for $r_j=\left( \frac{2}{\eta^\l} \right)^{k_j(1+\e)}$:
$$r_j^{\frac{\a}{1+\e}} = 2^{k_j} \leq s(w_{k_j}) \leq \log b_{\G_\o}(|w_{k_j}|) \leq \log b_{\G_\o}\left( \left( \frac{2}{\eta^\l} \right)^{(1+\e)k_j} \right)= \log b_{\G_\o}(r_j).$$
Interpolating for $r_j \leq r \leq r_{j+1}$ gives:
$$\log b(r) \geq \log b(r_j) \geq r_j^{\frac{\a}{1+\e}}=r_{j+1}^{\frac{\a}{1+\e}}\left( \frac{2}{\eta^\l} \right)^{-\a p_j} \geq r^{\frac{\a}{1+\e}} \left( \frac{2}{\eta^\l} \right)^{-\a p_j} \geq r^{\a-2\e}, $$
where the last inequality holds for large $r$ since $p_{j+1} \leq P(k)=o(k)=o(\log r)$. As $\e$ is arbitrary, $\underline{\a}(\G_\o) \geq \a$ for any such sequence $\o$.
\end{proof}

\begin{remark} 
Obviously, the computation of the exact growth exponent $\a(G)=\a$ does not imply that $b_G(r) \simeq e^{r^\a}$. The precise estimates obtained with the proof above are (for $r_j \leq r \leq r_{j+1}$):
$$C^{-(j+p_j+e(j))}r^\a \leq \log b_{\o}(r) \leq r^\a A^{p_{j+1}}, $$
where $e(j)= (\sum_{i=1}^jq_i)-\l k_j =o(j)$ is the error on the rationnal approximation of $\l$ by greater values. Taking $p_i$ of the order $\log i$, and thus $j$ of the order $\frac{\log r}{\log \log r}$, one obtains for some $A$:
$$ r^{\a- \frac{A}{\log\log r}} \leq \log b_\o(r) \leq r^{\a+\frac{A\log\log r}{\log r}}, $$
and taking $p_i$ of the order $i^\theta$ for $0<\theta<1$, thus $j$ of order $(\log r)^{\frac{1}{\theta+1}}$, one obtains:
$$r^{\a-A(\log r)^{-\frac{\theta}{\theta+1}}} \leq \log b_\o(r) \leq r^{\a+A(\log r)^{-\frac{1}{\theta+1}}}. $$
\end{remark}

\section{Oscillation phenomena}\label{section oscillation}

\subsection{Groups with oscillating logarithmic growth exponents} The oscillation of logarithmic exponents of growth function is the phenomenon that underlies the construction of antichains of growth function in section 7 of \cite{Gri1} and of ``fast intermediate'' growth in \cite{Ers3}. It was studied for its own interest in the second chapter of \cite{Bri}. Theorem \ref{precisegrowth} allows a better understanding.

\begin{theorem}\label{theoremoscillation}
For any $\a \leq \b \in [\a_0,1]$, there exists a sequence $\o(\a,\b)$ such that $\underline{\a}(\G_{\o(\a,\b)})=\a$ and $\overline{\a}(\G_{\o(\a,\b)})=\b$, i.e.
$$\liminf \frac{\log \log b_{\G_{\o(\a,\b)}}(r)}{\log r} = \a \textrm{ and } \limsup \frac{\log \log b_{\G_{\o(\a,\b)}}(r)}{\log r} = \b.$$
\end{theorem}
To ease notations, $b_{\G_\o}(r)=b_\o(r)$ from now on.
\begin{proof}
Take $\o(\a,\b)=\o(\a)_{|0\dots m_1}\o(\b)_{|m_1+1\dots n_1}\o(\a)_{|n_1+1\dots m_2}\o(\b)_{|m_2+1\dots n_2}\dots$ for some sequences $m_i,n_i$ tending to infinity. Such a choice ensures that:
$$\a \leq \underline{\a}(\G_{\o(\a,\b)}) \textrm{ and } \overline{\a}(\G_{\o(\a,\b)}) \leq \b. $$
If $m_i,n_i$ tend to infinity sufficiently fast, these inequalities become equalities. Indeed, take $\e_i \rightarrow 0$, and construct $r_i,r_i'$ such that:
$$\frac{\log \log b(r_i)}{\log r_i} \leq \a+\e_i \textrm{ and } \frac{\log \log b(r_i')}{\log r_i'} \geq \b-\e_i. $$
By localization \ref{localization}, left inequality holds for all $\o_{|0\dots m_i}=\o(\a,\b)_{|0\dots m_i}$ and right inequality for all $\o_{|0\dots n_i}=\o(\a,\b)_{|0\dots n_i}$ with $\log_2 r_i=m_i$ and $\log_2 r_i'=n_i$.

Assume by induction that $m_j,n_j$ are constructed for $j \leq i$ and construct $m_{i+1}=\log r_{i+1}$. Take $\o'=\o(\a,\b)_{|0\dots n_{i}}\o(\a)_{|n_{i}+1\dots}$. By proposition \ref{comparaisonfinie} on asymptotic growth:
$$b_{\o'}(r) \leq 2^{2^{n_i}}b_{\s^{n_i}\o'}(\frac{r}{2^{n_i}}+1)^{2^{n_i}} = 2^{2^{n_i}}b_{\s^{n_i}\o(\a)}(\frac{r}{2^{n_i}}+1)^{2^{n_i}} \leq 2^{2^{n_i}}b_{\o(\a)}(r+2^{n_i+1})^{2^{n_i}},$$ 
so that:
$$\frac{\log\log b_{\o'}(r)}{\log r} \leq \frac{\log \log b_{\o(\a)}(r+2^{n_i+1})+n_i\log 2}{\log r} \simeq \frac{\log\log b_{\o(\a)}(r)}{\log r} \longrightarrow_{r \rightarrow \infty} \a,$$
and there exists $r_{i+1}$ as required. Set $m_{i+1}=\log_2(r_{i+1})$.

Now construct $n_{i+1}= \log_2 (r_{i+1}')$. Take $\o''=\o(\a,\b)_{|0\dots m_{i+1}}\o(\b)_{|m_{i+1}+1\dots}$. Again proposition \ref{comparaisonfinie}:
$$b_{\o''}(r) \geq b_{\s^{m_{i+1}}\o''}(\frac{r}{2^{m_{i+1}}}-1) = b_{\s^{m_{i+1}}\o(\b)}(\frac{r}{2^{m_{i+1}}}-1) \geq \frac{1}{2}b_{\o(\b)}(r-2^{m_{i+1}+1})^{\frac{1}{2^{m_{i+1}}}}, $$
so that:
$$\frac{\log\log b_{\o_{i+1}'}(r)}{\log r} \geq  \frac{\log \log b_{\o(\b)}(r-2^{m_{i+1}+1})-m_{i+1}\log 2}{\log r} \longrightarrow_{r \rightarrow \infty} \b, $$
and there exists $r_{i+1}'$ and $n_{i+1}=\log_2 r_{i+1}'$.
\end{proof}

\subsection{Antichains of growth functions} The following result is an improvement of Theorem 7.2 in \cite{Gri1}, which established the existence of antichains of intermediate growth functions accumulating to $e^r$.

\begin{theorem}\label{antichain}
For any $\a_0 \leq \a < \b \leq 1$, there exists uncountably many groups $\G_\o$ with pairwise non comparable growth functions (such a collection of groups is called an antichain) satisfying $\underline{\a}(\G_\o)=\a$ and $\overline{\a}(\G_\o)=\b$.

Moreover, if $\b < \b' \leq 1$, such an antichain can be chosen so that $b_\o(r) \leq C e^{r^{\b'}}$ for a constant $C$ depending only on $\b,\b'$ and not on $\o$.
\end{theorem}

\begin{lemma}\label{lemmaantichain}
Given $\a_0 \leq \a < \b \leq 1$, there exists an application $\o$ from the set $\Fc(\N,\{\a,\b\})$ of functions $f: \N \rightarrow \{\a,\b\}$ to the Cantor space of infinite sequences $\{0,1,2\}^\N$, and there exists sequences $r_i \rightarrow \infty$ and $\frac{\b-\a}{2}>\e_i \rightarrow 0$ such that:
\begin{enumerate}
\item $\underline{\a}(\G_{\o(f)})=\a$ and $\overline{\a}(\G_{\o(f)})=\b$,
\item $\frac{\log\log b_{\o(f)}(r_i)}{\log r_i} \geq \b-\e_i$ if $f(i)=\b$,
\item $\frac{\log\log b_{\o(f)}(r_i)}{\log r_i} \leq \a+\e_i$ if $f(i)=\a$.
\end{enumerate}
\end{lemma}

\begin{proof}[Proof of theorem \ref{antichain}]
There are uncountably many functions $\x:\N \times \N \rightarrow \{\a,\b \}$ such that $\x(x,y)=\a$ implies $\x(x,y+1)=\b$ and $\x(x,y)=\b$ implies $\x(x,y+1)=\a$. Any bijection $\f:\N \rightarrow \N \times \N$, provides an injection $\x \mapsto f_\x$ by $f_\x(i)=\x \circ \f(i)$. Now given $\x_1 \neq \x_2$, if $f_{\x_1}(i) < f_{\x_2}(i)$, there exists $j >i$ such that $f_{\x_2}(j) < f_{\x_1}(j)$. Lemma \ref{lemmaantichain} ensures that $b_{\o(f_{\x_1})}(r)$ and $b_{\o(f_{\x_2})}(r)$ are not comparable.
\end{proof}

\begin{proof}[Proof of lemma \ref{lemmaantichain}]
The proof of this lemma is a variation on the proof of theorem \ref{theoremoscillation}. Pick:
$$\o(f)=\o(f(0))_{|0\dots m_0}\o(f(1))_{|m_0+1\dots m_1}\o(f(2))_{|m_1+1\dots m_2}\dots $$
for a sequence $m_i=\log_2(r_i)$ increasing sufficiently fast. Mind that this guarantees a uniform upper bound $\frac{\log\log b_{\o(f)}(r)}{\log r} \leq \b+\e=\b'$ for any $\e$ and $r$ big enough (depending on $\e$).

Assume by induction $m_j$ and $r_j$ constructed for $j \leq i$ and consider:
\begin{eqnarray*}
\o' &=& \o(f(0))_{|0\dots m_0}\dots \o(f(i))_{|m_{i-1}+1\dots m_i} \o(\a)_{|m_i+1\dots}, \\
\o'' &=& \o(f(0))_{|0\dots m_0}\dots \o(f(i))_{|m_{i-1}+1\dots m_i} \o(\b)_{|m_i+1\dots}. 
\end{eqnarray*}
As above, proposition \ref{comparaisonfinie} on asymptotic growth provides:
\begin{eqnarray*}
b_{\o'}(r) &\leq& 2^{2^{m_i}} b_{\o(\a)}(r+2^{m_i+1})^{2^{m_i}}, \\
b_{\o''}(r) &\geq& \frac{1}{2} b_{\o(\b)}(r-2^{m_i+1})^{\frac{1}{2^{m_i+1}}}, 
\end{eqnarray*}
so that there exists $r_{i+1}$, independent of $(f(0),\dots,f(i))$, such that:
\begin{eqnarray*}
\frac{\log \log b_{\o'}(r_{i+1})}{\log r_{i+1}} &\leq&  \a+\e_i, \\
\frac{\log \log b_{\o''}(r_{i+1})}{\log r_{i+1}} &\geq&  \b-\e_i,
\end{eqnarray*}
and this is true for any sequence $\o$ coinciding with $\o',\o''$ on the $m_{i+1}=\log_2 r_{i+1}$ first values.
\end{proof}

\begin{remark}
The idea behind the proof of theorem \ref{theoremoscillation}, is that the asymptotic behavior of the growth function $b_\o(r)$ of the group $\G_\o$ depends only on the asymptotic of the defining sequence $\o$, whereas locally a ball of given radius depends only on some first terms of $\o$. This permits to produce scales at which the growth function is essentially $e^{r^\a}$ and others at which it is essentially $e^{r^\b}$, thus explaining oscillation between this two behaviors. Of course, the process can be used to produce a variety of different behaviors at different scales, for instance scales $S_i$ at which $\G_\o$ seems to have growth $e^{r^{\g_i}}$ for countably many $\a_i \in [\a_0,1]$, intertwined with scales $S_j$ at which $\G_\o$ seems to have growth oscillating between $e^{r^{\a_j}}$ and $e^{r^{\b_j}}$. The only point is to allow enough ``time'' so that the behavior at scale $S_i$ or $S_j$ becomes visible, i.e. functions $m_i,n_i$ in the proofs above increasing sufficiently fast. 
\end{remark}

\section{Frequency of oscillations}

This section aims at studying the frequency of oscillations for groups of the type $\G_\o$. The main question is to maximize the frequency of oscillation between two given bounds, or equivalently to minimize the period. Some pseudo-period exponents are defined, and shown to be group invariants. Some estimates on these exponents are then obtained.

\subsection{Group invariants associated to oscillation}

Given $\a <\b$ and a Lipschitz function $b:\N \rightarrow \N$, define the {\it upper set} $U(\a,\b)$ and {\it lower set} $L(\a,\b)$ of $b$ for $\a,\b$ to be:
$$U(\a,\b)=\{s \in \N| \frac{\log \log b(s)}{\log s} \geq \b\} \textrm{ and } L(\a,\b)=\{t \in \N| \frac{\log \log b(t)}{\log t} \leq \a\}. $$
Note that $\frac{\log \log b(s)}{\log s} \geq \b$ is equivalent to $\log b(s) \geq s^{\b}$ and $\frac{\log \log b(t)}{\log t} \leq \a$ is equivalent to $\log b(t) \leq t^{\a}$.

\begin{property}
Let $\a<\b$ and $b:\N \rightarrow \N$ be a Lipschitz function, then:
\begin{enumerate}
\item $L(\a,\b) \sqcup U(\a,\b) \subset \N$, and the inclusion is strict if both upper and lower sets are infinite.
\item Assume $\a'<\a<\b<\b'$ then:
$$L(\a',\b) \subset L(\a,\b) \textrm{ and } U(\a,\b') \subset U(\a,\b), $$
and the inclusions are strict if both upper and lower sets are infinite.
\end{enumerate}
\end{property}

Note that when $b(r)=b_\G(r)$ is the growth function of a finitely generated group $\G$ such that $\underline{\a}(\G) < \a < \b < \overline{\a}(\G)$, then both upper and lower sets are infinite.

Property (1) allows to decompose $U= \bigsqcup_{j=0}^\infty U_j$ and $L= \bigsqcup_{j=0}^\infty L_j$ such that:
\begin{enumerate}
\item $U_i,L_i$ are non empty,
\item for any $s \in U_i$, then $s \geq \max \cup_{j \leq i-1}L_j$ and $s \leq \min \cup_{j \geq i} L_j$,
\item for any $t \in L_i$, then $t \geq \max \cup_{j \leq i} U_j$ and $t \leq  \min \cup_{j \geq i+1} U_j$.
\end{enumerate}
Call this decomposition alternating (see figure \ref{fig2}).

\begin{figure}
\begin{center}

\psset{algebraic=true}
\begin{pspicture}(-0.5,0)(14.5,7)
\psline{->}(-0.5,0.5)(14.5,0.5)
\psline{->}(0,0)(0,7)
\pscurve(0.2,3)(2.5,6.5)(6.8,1.5)(11,6.5)(12,5.8)(13,6.3)(14.5,5.6)
\psline[linestyle=dashed](1.6,6)(1.6,0.5)
\psline[linestyle=dashed](3.6,6)(3.6,0.5)
\psline[linestyle=dashed](10.1,6)(10.1,0.5)
\psline[linestyle=dashed](13.85,6)(13.85,0.5)
\psline[linestyle=dashed](11.65,6)(11.65,5)
\psline[linestyle=dashed](12.5,6)(12.5,5)
\psline[linestyle=dashed](5.75,2)(5.75,0.5)
\psline[linestyle=dashed](7.9,2)(7.9,0.5)
\psplot{0}{14.5}{2}
\psplot{0}{14.5}{6}
\rput*(-0.4,6){$\b$}
\rput*(-0.4,2){$\a$}
\rput*(1.6,0){$s_i$}
\rput*(3.6,0){$s_i'$}
\rput*(10.1,0){$s_{i+1}$}
\rput*(13.7,0){$s_{i+1}'$}
\rput*(5.75,0){$t_i$}
\rput*(7.9,0){$t_i'$}
\rput*(2.6,5.5){$U_i$}
\rput*(6.8,2.5){$L_i$}
\rput*(10.9,5.5){$U_{i+1}$}
\rput*(13.2,5.5){$U_{i+1}$}
\end{pspicture}

\end{center}
\caption{\label{fig2} Upper and lower sets $U(\a,\b)$ and $L(\a,\b)$ seen by drawing the curve $f(r)=\frac{\log \log b(r)}{\log r}$. }
\end{figure}
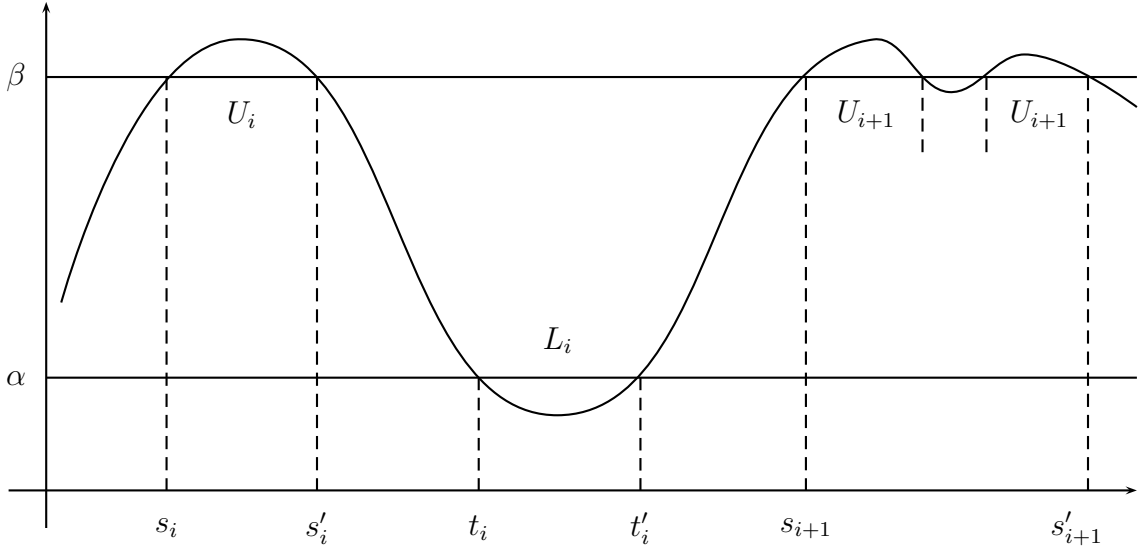

In order to study oscillation, set $s_i=\min U_i$, $s_i'=\max U_i$, $t_i=\min L_i$ and $t_i'=\max L_i$. The {\it upper pseudo period function} $u$ is the partially defined $s_{i+1}=u(t_i')$ and the {\it lower pseudo period function} $l$ is the partially defined $t_i=l(s_i')$. In order to investigate how small these functions can be, define:
$$u_{\a,\b}=\inf \{\n| \exists i_o,\forall i \geq i_o, s_{i+1} \leq (t_i')^\n\} \textrm{ and } l_{\a,\b}=\inf \{\l| \exists i_o,\forall i \geq i_o, t_i \leq (s_i')^\l\}.$$
Equivalently:
$$u_{\a,\b}= \limsup_{i \rightarrow \infty} \frac{\log s_{i+1}}{\log t_i'} \textrm{ and } l_{\a,\b}= \limsup_{i \rightarrow \infty} \frac{\log t_i}{\log s_i'}.$$

The following fact provides estimates on the pseudo period functions that any growth function of infinite group must satisfy.

\begin{fact}\label{trivial}
Consider $\a <\b$ and a function $b:\N \rightarrow \N$, then:
\begin{enumerate}
\item if $b(r)$ is submultiplicative, $u_{\a,\b} \geq \frac{1-\a}{1-\b}>1$,
\item if $b(r)$ is increasing, $l_{\a,\b} \geq \frac{\b}{\a} >1$.
\end{enumerate}
\end{fact}

\begin{proof}
Suppose $\log b(t) \leq t^{\a}$. Submultiplicativity implies $\log b(kt) \leq k t^{\a}$, so that $\log b(kt) \geq (kt)^{\b}$ forces $kt^{\a} \geq (kt)^{\b}$ hence $kt \geq t^{\frac{\b-\a}{1-\b}+1}$. Now suppose $\log b(s) \geq s^{\b}$, then $\log b(t) \leq t^{\a}$ forces $t^{\a} \geq s^{\b}$.
\end{proof}

By property (2), given $\a''<\a'<\b'<\b''$, one has $u_{\a'',\b''} \geq u_{\a',\b'}$ and $l_{\a'',\b''} \geq l_{\a',\b'}$. This permits the:

\begin{definition}
The {\it upper pseudo period exponent} $u(\a,\b)$ and the {\it lower pseudo period exponent} $l(\a,\b)$ of a function $b(r)$ are: 
$$u(\a,\b)=\lim_{\begin{array}{c} \a' \rightarrow \a^+ \\ \b' \rightarrow \b^- \end{array}} u_{\a',\b'}, \textrm{ and } l(\a,\b)=\lim_{\begin{array}{c} \a' \rightarrow \a^+ \\ \b' \rightarrow \b^- \end{array}} l_{\a',\b'}. $$
Remind notation $\a' \rightarrow \a^+$ (respectively $\b' \rightarrow \b^-$) for $\a' \rightarrow \a$ and $\a'>\a$ (respectively $\b' \rightarrow \b$ and $\b' < \b$).
\end{definition}

This definition is appropriate because it permits to define $u(\underline{\a}(\G),\overline{\a}(\G))$ and $l(\underline{\a}(\G),\overline{\a}(\G))$ associated to the growth function $b_\G(r)$ even though the upper and lower sets $U(\underline{\a}(\G),\overline{\a}(\G))$ and $L(\underline{\a}(\G),\overline{\a}(\G))$ may be empty. Also:

\begin{proposition}
The upper and lower pseudo period exponents $u(\a,\b)$ and $l(\a,\b)$ are group invariants.
\end{proposition}

\begin{proof}
In order to show the exponents are not perturbed by change of generating set, consider a function $b'(r)$ such that there exists $C$ with $b(\frac{r}{C}) \leq b'(r) \leq b(Cr)$.

Then $L^{(b')}(\a',\b')=\{t|\log b'(t) \leq t^{\a'}\} \subset \{t|\log b(\frac{t}{C}) \leq t^{\a'}\} = C\{x| \log b(x) \leq C^{\a'}x^{\a'}\}$. But given any $\a''>\a'$ and $x$ large enough, one has $C^{\a'}x^{\a'} \leq x^{\a''}$, so that if $x$ large enough belongs to $L^{(b')}(\a',\b')$, then $x$ belongs to $CL^{(b)}(\a'',\b'')$. Similarly, for any $\b''<\b'$, large enough $y$ that belong to $U^{(b')}(\a',\b')$ also belong to $\frac{1}{C}U^{(b)}(\a'',\b'')$.

This permits to deduce that there is a $j$ such that:
$$\frac{\log s_{i+1}^{(b')}(\a',\b')}{\log t_i^{'(b')}(\a',\b')} \geq \frac{\log s_{j+1}^{(b)}(\a'',\b'')-\log C}{\log t_j^{'(b)}(\a'',\b'')+\log C} $$
so that $u^{(b')}_{\a',\b'} \geq u^{(b)}_{\a'',\b''}$ for any $\a'<\a''<\b''<\b'$, which implies $u^{(b')}(\a,\b) \geq u^{(b)}(\a,\b)$, and equality holds by symetry. Similar proof for $l(\a,\b)$.
\end{proof}

\begin{remark}\label{period}
Given $\a < \b$, one can similarly define the pseudo period exponent of oscillations for a function $b(r)$, by $p(\a,\b)=\lim p_{\a',\b'}$ for $p_{\a',\b'}=\limsup \frac{\log s_{i+1}}{\log s_i}$, and it is a group invariant for $b_\G(r)$. However, it is not true a priori that replacing $s_i$ by $t_i$, $t'_i$ or $s'_i$ would provide the same exponent. 
\end{remark}

\subsection{Estimates on pseudo-period}

Theorem \ref{precisegrowth} shows that for any $\g \in [\a_0,1]$ there exists a group $\G_{\o(\g)}$ such that:
$$\frac{1}{C_\e} r^{\g-\e} \leq \log b_{\o(\g)}(r) \leq C_\e r^{\g+\e}, $$
where $\e>0$ is arbitrary and $C_\e$ depends only on $\e$.

Suppose that $\log b_\o(t)=t^\a$ for some $t$. This fact depends only on $(\o_i)_{i=0}^{m}$ for $m=\log_2 t$ by localization. Now consider the group $\G_{\o'}$ for the sequence $\o'=\o_0\dots \o_m \o(\g)|_{m+1 \dots}$, for some $\g \geq \b >\a$. By proposition \ref{comparaisonfinie} on asymptotic growth, one has:
$$\log b_{\o'}(s) \geq \frac{1}{2^m}\log b_{\o(\g)}(s-2^{m+1}-\log 2) \geq \frac{1}{C_\e t}(s-2t)^{\g-\e}, $$
so that for any $\b'<\b$ and $\e$ small enough:
$$\min\{s|\log b_{\o'}(s) \geq s^{\b'}\} \leq C_\e t^{\frac{1}{\g-\e-\b'}}+o(t^{\frac{1}{\g-\e-\b'}}). $$

Conversely suppose that $\log b_\o(s)=s^\b$ for some $s$, which depends only on $(\o_i)_{i=0}^{n}$ for $n=\log_2 s$, and consider the group $\G_{\o''}$ for the sequence $\o''=\o_0\dots \o_n \o(\d)|_{n+1 \dots}$ for some $\d \leq \a < \b$. As above, one has:
$$\log b_{\o''}(t) \leq 2^n(\log b_{\o(\d)}(t+2^{n+1})+\log 2) \leq C_\e s (t+2s)^{\d+\e}, $$
so that for any $\a < \a'$ and $\e$ small enough:
$$\min \{t| \log b_{\o''}(t) \leq \t^{\a'} \} \leq C_\e s^{\frac{1}{\a'-\d-\e}}+o(s^{\frac{1}{\a'-\d-\e}}). $$
These two observations show the following (passing to the limits $\a' \rightarrow \a$, $\b' \rightarrow \b$ and $\e \rightarrow 0$):

\begin{proposition}\label{periodbounds}
Given $\a_0 \leq \d \leq \a < \b \leq \g \leq 1$, there exists a sequence
$$\o(\a,\b,\g,\d)=\o(\d)|_{0\dots m_1}\o(\g)|_{m_1+1\dots n_1}\o(\d)|_{n_1+1\dots m_2}\o(\g)|_{m_2+1\dots n_2}\dots$$
such that the group $\G_{\o(\a,\b,\g,\d)}$ satisfies:
$$u(\a,\b) \leq \frac{1}{\g-\b} \textrm{\quad and \quad} l(\a,\b) \leq \frac{1}{\a-\d}. $$
\end{proposition}

The choice of $\o(\a,\b,\g,\d)$ guarantees $\underline{\a}(\G_{\o(\a,\b,\g,\d)}) \geq \d$ and $\overline{\a}(\G_{\o(\a,\b,\g,\d)}) \leq \g$, but these are probably strict inequalities.

Note that the construction of $\o(\a,\b)$ in the proof of theorem \ref{theoremoscillation} is a particular instance of the above proposition with $\g=\b$ and $\d=\a$. In this case, the upper and lower pseudo period exponents are (a priori) infinite.

On the other hand, in order to minimize the upper and lower pseudo period exponents for a fixed oscillation magnitude $\a<\b$, taking $\g=1$ and $\d=\a_0$ gives upper bounds (the lower bounds are trivial from fact \ref{trivial}):
$$\frac{1-\a}{1-\b} \leq u(\a,\b) \leq \frac{1}{1-\b} \textrm{\quad and \quad} \frac{\b}{\a} \leq l(\a,\b) \leq \frac{1}{\a-\a_0}. $$

Since the estimates above are done for any $t$ in $L(\a,\b)$ and $s$ in $U(\a,\b)$, they provide an upper bound for (any choice of definition in remark \ref{period}) pseudo period:
$$\frac{\b(1-\a)}{\a(1-\b)} \leq p(\a,\b) \leq \frac{1}{(1-\b)(\a-\a_0)}. $$

\subsection{Further developments on frequency of oscillations} After this work was first submitted, the results of \cite{BE2} show that there are groups with $u(\a,\b)=\frac{1-\a}{1-\b}$ and $l(\a,\b)=\frac{\b-\a_0}{\a-\a_0}$. Such an upper pseudo-period exponent is optimal by submultiplicativity of growth function. It is not known if such a lower pseudo-period exponent is optimal. 

Similar results concerning oscillations of entropy functions of random walks on groups are obtained in \cite{Bri3}, theorem 5.8.

\end{document}